    	\documentclass[11pt]{amsart}

\usepackage{amsrefs}

     \usepackage[all]{xy}

     \makeatletter
\def\setzerodotbox#1{%
   \setbox\zerodotbox@=\hbox{%
       \dimen@=#1\xydashw@
       \kern-\dimen@ \vrule width2\dimen@ height\dimen@ depth\dimen@
   }%
}
\makeatother

     \usepackage{enumerate}

     \newenvironment{problist}
{\begin{enumerate}[(a)]}
{\end{enumerate}}

     	\usepackage{amssymb}
     \usepackage{amsmath}
    	\usepackage{amsfonts}
	\usepackage{latexsym}   		
	
    \newtheorem{thm}{Theorem}

    \newtheorem{lem}[thm]   {Lemma}
    \newtheorem{cor}[thm]   {Corollary}

    \newtheorem{prop}[thm]  {Proposition}

    \newcommand{\term}[1]   {{\bf #1}\index{#1}}

\newcommand{\?}{\, ?\, }
\newcommand{\st}{\mid}

    \newcommand{\A}        {\mathcal{A}}

                \newcommand{\K}        {\mathcal{K}}

\newcommand{\R}  {\mathcal{R}}

\newcommand{\U}{\mathcal{U}}

    \newcommand{\ZZ}    {\mathbb{Z}}

    \newcommand{\NN}    {\mathbb{N}}
    \newcommand{\FF}    {\mathbb{F}}

    \newcommand{\s}     {\Sigma}
    
    \newcommand{\smsh}  {\wedge}
    \newcommand{\om}    {\Omega}
    
    \newcommand{\wdg}   {\vee}

    \newcommand{\twdl}  {\widetilde}

    \newcommand{\sseq}  {\subseteq}

\DeclareMathOperator{\limone}{lim^1}

            \DeclareMathOperator{\cl}    {cl}

 \DeclareMathOperator{\im}    {Im}

    \DeclareMathOperator{\Hom}   {Hom}
    \DeclareMathOperator{\Ext}   {Ext}

 \DeclareMathOperator{\Ph}    {Ph}
    
        \DeclareMathOperator{\map}    {map}

\newcommand{\WedgesOfSpheres}{\mathcal{W}}
	
\begin{document}
	
	


\title{A Simple Homotopy-Theoretical 
Proof of the Sullivan Conjecture}


\author{Jeffrey Strom}
\address{Department of Mathematics\\
Western Michigan University\\
Kalamazoo, MI\\
49008-5200 USA}
\email{Jeff.Strom@wmich.edu}	




\begin{abstract}
We give a new proof, using comparatively 
simple techniques, 
of the Sullivan conjecture:  
$\map_*(B\ZZ/p, K) \sim *$ for every 
finite-dimensional CW complex $K$.
\end{abstract}

\maketitle
\pagestyle{myheadings}
\markboth{{J. Strom}}
{{A simple 
homotopy-theoretical
proof of the Sullivan conjecture}}


\section*{Introduction}

Haynes Miller proved the Sullivan conjecture in \cite{MR750716}.
Another proof
can be deduced from the extension due to Lannes
\cite{MR1179079}, 
whose proof depends on the insights into unstable 
modules and algebras afforded by the $T$ functor.
There are three main problems with these proofs:  
 they are very complicated;  
the content is almost entirely encoded in pure algebra; and 
  it is difficult to tease out the fundamental properties of
$B\ZZ/p$ that make the proofs work.

Our purpose in this paper is  
to offer a new proof avoiding these complaints.
The Sullivan conjecture is an 
easy consequence of the following 
main theorem.

\begin{thm}
\label{thm:main}
Let $X$ be a CW complex of finite type, and assume 
that $\twdl H^*(X; \ZZ[{1\over p}]) =0$.
Then each of the following conditions implies the 
next.
\begin{enumerate}
\item
$\twdl H^*(X; \FF_p)$ is a reduced unstable $\A_p$-module
and 
$\twdl H^*(X; \FF_p) \otimes J(n)$ is an injective unstable 
$\A_p$-module for   all  $n\geq 0$.
\item
$\Ext_\U ^s ( \s^{2m+1}\FF_p, \twdl H^*(\s^{s+t} X) ) = 0$
for every $s, t\geq 0$ and all   $m\geq 0$.
\item
$\map_*(X, S^{2m+1} ) \sim *$ for 
all sufficiently large  $m$.
\item
$\map_*(X, K) \sim *$ for all simply-connected finite complexes $K$.
\item
$\map_*(X, \bigvee_{i=1}^\infty S^{n_i} ) \sim *$ 
for any countable   set    $\{ n_i \}$ with each $n_i > 1$.
 \item
$\map_*( X, K) \sim *$ for all simply-connected
finite-dimensional CW complexes $K$. 
\item
$\map_*( X, K) \sim *$ for all simply-connected CW complexes
 $K$ with $\cl_{\WedgesOfSpheres}(K)< \infty$.
\end{enumerate}
Furthermore, if  
 $\pi_1(X)$ has no perfect quotient groups (that is, if $\pi_1(X)$ is 
 \term{hypoabelian}), then 
the simply-connected hypotheses on $K$ are not needed.
\end{thm}

The notation $\cl_{\WedgesOfSpheres}(K)$ in part (7) denotes the cone length
of $K$ with respect to the collection $\WedgesOfSpheres$ 
of all wedges 
of spheres; see \cite[\S 1]{MR2029919} for a brief overview
of the main properties of cone length.  
Since (7) implies (3), the last five statements 
are actually equivalent.

Every finite-dimensional CW complex $K$
of course has finite cone length with respect to $\WedgesOfSpheres$.
The extension from part (6) to part (7) is a 
geometric parallel to the passage from 
finite-dimensional spaces in \cite{MR750716}
to spaces with locally finite cohomology in \cite{MR1179079}.

Since it is known that 
 the reduced cohomology 
 $\twdl H^*(B\ZZ/p; \FF_p)$ is reduced 
 \cite[Lem. 2.6.5]{MR1282727}
 and
 that $\twdl H^*(B\ZZ/p; \FF_p) \otimes J(n)$ is an injective unstable 
$\A_p$-module for   all  $n\geq 0$
\cite[Thm. 3.1.1]{MR1282727}, 
we  obtain the Sullivan conjecture as an immediate consequence.

\begin{cor}[Miller]
If $\cl_{\WedgesOfSpheres}( K ) < \infty$ (and, in particular, if
$K$ is  finite-dimensional), then 
$
\map_*(B\ZZ/p, K) \sim *
$.
\end{cor}

The claim that our proof is \textit{simple} should be justified:
we make no use whatsoever of spectral sequences, 
except implicitly in making use of the well-known 
cohomology of Eilenberg-Mac Lane spaces;
 the existence of
unstable \textit{algebras} over the Steenrod algebra is
mentioned only to make sense of Massey-Peterson towers; 
finally, the only homological algebra in this paper is the usual 
abelian kind.  
It has  seemed appropriate at points to emphasize the 
simplicity of the present approach by laying out some results
that might just as well have been cited from their sources.

%
%

\medskip

\paragraph{\bf Acknowledgement}
Many thanks are due to John Harper for pointing out that I had the raw 
materials for a proof of the full Sullivan conjecture,  and for 
bringing the paper \cite{MR764593} to my attention.

\section{Preliminaries}

We begin by reviewing some preliminary material on
the category $\U$ of unstable $\A_p$-algebras, 
  Massey-Peterson towers and  phantom maps.

\subsection{Unstable Modules over the Steenrod Algebra}

\label{subsection:U}

The cohomology functor $H^*(\?; \FF_p)$
takes its values in the category $\U$ of   unstable
modules and their homomorphisms.  
An \term{unstable module} over the Steenrod algebra $\A_p$
is a graded $\A_p$-module $M$ satisfying   $P^I(x) = 0$ if 
$e(I) > |x|$, where $e( I )$ is
the excess of $I$ and $|x|$ is the degree
of $x\in M$.     
%
We begin with some basic algebra of unstable modules, 
all of which is (at least implicitly) in \cite{MR1282727}.

\medskip

\paragraph{\bf Suspension of Modules}
An unstable module $M\in \U$ has a \term{suspension}
 $\s M\in \U$
given by $(\s M)^n  = M^{n-1}$.  The functor $\s : \U \to \U$
has a left adjoint $\om$ and a right adjoint $\twdl \s$.  A module
$M$  
is called \term{reduced} if $\twdl \s M = 0$.

\medskip

\paragraph{\bf Projective and Injective Unstable Modules}
In the category $\U$, there are free modules
$F(n) = \A_p/ E(n)$, where $E(n)$ is the smallest left
ideal containing all Steenrod powers $P^I$
with excess $e(I) > n$.   It is easy to see that 
the assignment $f\mapsto f([1])$ defines 
natural isomorphisms
\[
\Hom_\U (F(n) , M) \xrightarrow{\, \cong\, } M^n .
\] 
This property defines $F(n)$ up to 
natural isomorphism, and shows that $F(n)$ deserves to 
be called a \term{free} module on a single generator of
dimension $n$.  More generally, the free module on 
a set $X= \{ x_\alpha\}$ with $|x_\alpha| = n_\alpha$
is (up to isomorphism) the sum $\bigoplus F(n_\alpha)$
(see \cite[\S 1.6]{MR1282727} for details).
 
 A graded $\FF_p$-vector space $M$ 
is   of \term{finite type} if  
  $\dim_{\FF_p}( M^k) < \infty$  for each $k$.
  Since $\A_p$ is of finite type, so is $F(n)$.
 
 The functor   which takes $M\in \U$ and returns 
the dual $\FF_p$-vector space $(M^n)^*$ is representable:
there is a module $J(n) \in \U$ and a natural isomorphism
\[
\Hom_\U ( M, J(n) ) \xrightarrow{\, \cong\, }
\Hom_{\FF_p}( M^n,\FF_p) .
\]
Since finite sums of vector spaces are also finite products, 
these functors are exact, so the module $J(n)$ is an injective
object in $\U$.

 \medskip
 
 \paragraph{\bf The Functor $\overline \tau$}
 
 In \cite[Thm 3.2.1]{MR1282727} it is shown that 
 for any module $H \in \U$, the functor $H \otimes_{\A_p} \?$
 has a left adjoint, denoted $(? : H)_\U$.    
Fix a module $H$ (to  stand in for $\twdl H^*(X; \FF_p)$)
 and write $\overline \tau$ for the
 functor $(\? :H)_\U$; this is intended to evoke the 
 standard notation $\overline T$ for the special case $H = \twdl H^*(
 B\ZZ/p; \FF_p)$.

\begin{lem}
\label{lem:main}
Let $H\in \U$ be a reduced unstable
module of finite type
and suppose that $H\otimes J(n)$ is injective in $\U$
for every $n\geq 0$.
Then
\begin{problist}
\item
$\overline \tau$ exact,
\item
$\overline \tau$ commutes with suspension,
\item
if $M$ is free and of finite type, then so is $\overline \tau(M)$, and 
\item
if $H^0 = 0$, then  $\overline \tau(M) = 0$
for any finite module $M\in \U$.
\end{problist}
\end{lem}

\begin{proof}
These results are covered in 
Sections 3.2 and 
 3.3 of \cite{MR1282727}.
 Specifically, parts (a) and (b) are proved as in
 \cite[Thm. 3.2.2 \& Prop. 3.3.4]{MR1282727}.  
 Parts (c) and (d) may be proved following
  \cite[Lem. 3.3.1 \& Prop. 3.3.6]{MR1282727}, but since there are
  some changes needed, we prove those parts here.
%
%

Write $d_k  = \dim_{\FF_p}( H^k)$;
then there are natural isomorphisms
\begin{eqnarray*}
\Hom_\U \left( 
\overline \tau( F(n)) ,   M \right
) 
&\cong&
\Hom_\U \left( 
  F(n) , H \otimes M \right
) 
\\
&\cong &
\Hom_\U \left( \mbox{$\bigoplus_{i+j = n} 
F(i)^{\oplus d_{j}}
,  M$} 
\right) , 
\end{eqnarray*}
proving (c) in the case of a free module on one generator.  
Since $\overline \tau$ is a left adjoint, it
commutes with colimits (and sums in particular), 
we derive the full statement of (c).

If $H^0 = 0$, then $d_0 = 0$ and   $\overline \tau(F(n))$
is a sum of free modules $F(k)$ with $k < n$. 
Since $F(0) = \FF_p$, we see  that  $\overline \tau(\FF_p) = 0$;
then  (a), together with the fact that $\overline \tau$ commutes
with colimits,  implies that $\overline \tau (M) = 0$ for all trivial modules 
 $M$.    Finally, any finite module $M$ has filtration all of whose
 subquotients are trivial, and (d) follows.
\end{proof}

 \subsection{Massey-Peterson Towers}

Cohomology of spaces has more structure than just 
that of an unstable $\A_p$-module.  It has a cup product
which makes $H^*(X; \FF_p)$ into an 
\term{unstable algebra} over $\A_p$.  The category 
of unstable algebras is denoted $\K$.

The forgetful functor $\K\to \U$ 
has a left adjoint $U: \U \to \K$.  A space $X$  is said to have
\term{very nice} cohomology if $H^*(X) \cong U(M)$ for some
unstable module $M$ of finite type.

Since $U(F(n)) \cong H^*(K(\ZZ/p, n))$, 
there is a contravariant functor $K$ which carries a free module
$F$ to a generalized Eilenberg-Mac Lane space (usually abbreviated
GEM) $K(F)$ such that 
$H^*(K(F)) \cong U(F)$.   If $F$ is free, then so is $\om F$, and 
  $K(\om F) \simeq \om K(F)$.

\begin{lem}
\label{lem:Maps2KF}
For any $X$, 
$[ X, K(F)] \cong \Hom_\U( F, \twdl H^*(X))$.
\end{lem}

It is shown in \cite{MR764593,MR546361,MR0226637} 
that if $H^*(Y) \cong U(M) $   and $P_*\to M\to 0$
is a free resolution in $\U$, 
then $Y$ has a \term{Massey-Peterson tower}
\[
\xymatrix{
\cdots \ar[r] &
Y_{s}\ar[r] \ar[d]& 
Y_{s-1} \ar[r] \ar[d] &
\cdots \ar[r] &
Y_1 \ar[r] \ar[d]&
Y_0\ar[d]
\\
& K(\om^s P_{s+1}) & K(\om^{s-1} P_s) && K(\om P_2) & K(P_1)
}
\]
in which 
\begin{enumerate}
\item
$Y_0 = K(P_0)$, 
\item 
each homotopy group $\pi_k(Y_s)$ is a finite $p$-group, 
\item
the limit of the tower is the completion $Y^\smsh_p$, 
\item
each sequence $Y_s \to Y_{s-1}\to K(\om^{s-1} P_s)$
is a fiber sequence, and
\item
the compositions   
$
\om K(\om^{s-1} P_s) 
\longrightarrow
Y_s
\longrightarrow
 K(\om^s P_{s+1})
$
 can be naturally
 identified with $K( \om^s d_{s+1} )$, where 
$d_{s+1}: P_{s+1}\to P_s$ is the differential
 in the given free resolution.
\end{enumerate}

\subsection{Phantom Maps}

A \term{phantom map} is a map $f:X\to Y$ from a CW complex
$X$ such that the restriction $f|_{X_n}$ of $f$ to the $n$-skeleton
is trivial for each $n$.    
We write $\Ph(X, Y) \sseq [X, Y]$ for the set of 
pointed homotopy classes
of phantom maps from $X$ to $Y$.
See \cite{MR1361910} for an excellent survey on phantom maps.

If $X$ is the homotopy colimit of a telescope diagram
$\cdots \to X_{(n)}\to X_{(n+1)}\to \cdots$, then 
there is a short exact sequence
of pointed sets 
\[
* \to \limone [\s X_{(n)}, Y] 
\longrightarrow
 [X, Y] 
 \longrightarrow
  \lim [ X_{(n)}, Y]\to *, 
\]
and dually, if $Y$ is the homotopy limit of a tower
$\cdots \gets Y_{(n)}\gets Y_{(n+1)}\gets \cdots$, then there
is a short exact sequence
\[
* \to \limone [ X, \om Y_{(n)}] 
\longrightarrow
 [X, Y] 
\longrightarrow
 \lim [ X, Y_{(n)}]\to * .
\]
In the particular case of the expression of a CW complex $X$ as
the homotopy colimit of its skeleta or of a space $Y$ as the
homotopy limit of its Postnikov system, the kernels are the
phantom sets.

%
%
%


We will be interested in showing that all phantom maps are trivial.
One useful criterion is that if $G$ is a tower of compact Hausdorff
topological spaces, then $\limone G = *$ 
(see \cite[Prop. 4.3]{MR1361910}).  This is used to prove
the following lemma.

\begin{lem}
\label{lem:CompactLimOneVanish}
Let $\cdots \gets Y_s \gets Y_{s+1} \gets \cdots$ be a tower 
of spaces such that
each homotopy group $\pi_k(Y_s)$ is finite.
If $Z$ is of finite type, then 
$
\limone [ Z, \om Y_s ] = *
$.
\end{lem}

\begin{proof}
The homotopy sets $[Z_n, \om^j Y_s]$ are  finite, and 
we give them the discrete topology, resulting in  towers of 
compact groups and continuous homomorphisms.
Fixing $s$ and lettting $n$ vary, we find that
 $\limone [ Z_n , \om^2 Y_s] = *$, 
 and hence the exact sequence
 \[
0 \to \limone_n [ Z_n, \om^2 Y_{s}] 
\longrightarrow
 [Z, \om Y_s] 
\longrightarrow
{ \lim}_n [ Z_n , \om Y_s]\to 1 
\]
(of groups)
reduces to an isomorphism 
$[Z , \om Y_s] \cong 
{ \lim}_n [ Z_n , \om Y_s]$.
Since $[Z , \om Y_s]$
 is an inverse limit of finite discrete spaces,  
it is compact and Hausdorff; and since the structure maps 
$Y_s \to Y_{s-1}$ induce maps of the towers that define the topology, 
    the induced maps
$ [Z , \om Y_s]\to [Z , \om Y_{s-1}]$ are continuous.
Thus
 $\limone [Z , \om Y_s] = *$.
\end{proof}
 

%
%
%
 
%

The Mittag-Leffler condition is another
useful criterion for the vanishing of $\limone$.
 A
tower of groups $\cdots \gets G_n \gets G_{n+1}\gets \cdots$
is \term{Mittag-Leffler} if there is a function $\kappa : \NN\to \NN$
such that for each $n$
 $
 \im( G_{n+k} \to G_n) = \im( G_{n+\kappa(n)} \to G_n)\sseq G_n
 $
 for  every $k \geq \kappa(n)$.
 
\begin{prop}
Let $\cdots \gets G_n \gets G_{n+1}\gets \cdots$ be a tower of groups.\begin{problist}
\item 
If the tower  is Mittag-Leffler, then $\limone  G_n = *$.
\item
If each $G_n$ is a countable group, then  the converse
holds:  if $\limone G_n = *$, then the tower is Mittag-Leffler 
\cite[Thm. 4.4]{MR1361910}.
\end{problist}
\end{prop}

%
%

Importantly, the Mittag-Leffler condition does not refer to the
algebraic structure of the groups $G_n$.
This observation plays a key role in 
 the following   result (cf. \cite[\S 3]{MR1357793}).

\begin{prop}
\label{prop:SameLoop}
Let  $X$ be a CW complex of finite type, and let
$Y_1$ and $Y_2$ be countable CW complexes 
with $\om Y_1 \simeq \om Y_2$.
Then 
$
\Ph( X, Y_1)  = * 
$
if and only if
$
\Ph(X, Y_2) = *
$.
\end{prop}

\begin{proof}
The homotopy equivalence
 $\om Y_1\simeq \om Y_2$ gives 
 levelwise bijections
$
\{ [\s X_n, Y_1]\}
\cong
\{ [ X_n, \om Y_1]\}
\cong 
\{ [ X_n,\om  Y_2]\}
\cong
\{ [\s X_n, Y_2]\}
$
of towers of sets. 
Since $X$ is of finite type and $Y_1, Y_2$ are countable
CW complexes, these towers are towers of countable 
groups.
Now the  triviality of the first phantom set implies that 
  the first tower is Mittag-Leffler; but then all four towers
 must be   Mittag-Leffler, 
and the result follows.
\end{proof}

%
%
%
%
%

\section{(1) implies (2)}

Write $H = \twdl H^*(X;\FF_p)$; thus
   $H\in \U$ is a reduced module of finite type 
and   $H\otimes J(n)$ is injective for all $n$.
If $P_*\to M\to 0$ is a free resolution of $M$ in $\U$, 
then  Lemma \ref{lem:main} implies that 
$\overline \tau(P_*) \to 0 \to 0$ is   a free resolution of $0$, so  
\begin{eqnarray*}
\Ext_\U^s \left( M, \s^{s+t} H \right)
&=&
\Ext_\U^s \left( M, H \otimes \s^{s+t}\FF_p)\right)
\\
&=&
H^s\left(  \Hom\left( P_*, H \otimes \s^{s+t}\FF_p\right)\right)
\\
&=&
H^s\left( \Hom \left( \overline \tau( P_*), \s^{s+t} \FF_p \right)\right)
\\
&=&
\Ext_\U^s \left( 0, \s^{s+t}\FF_p\right)
\\
&=& 0.
\end{eqnarray*}

\section{(2) implies (3)}


\begin{thm}
\label{thm:FakeHarper}
Suppose  $H^*(Y) = U(M)$ for some finite $M\in \U$,
and  
$Z$ is a CW complex
 of finite type with $\twdl H^*(Z; \ZZ[{1\over p}]) = 0$.
If    
$
\Ext_\U^s (M , \s^{s} \twdl H^*(Z)) =0 
$
for all $s\geq 0$, then $[   Z, Y ] = *$.
\end{thm}

Condition (2) allows us to apply Theorem \ref{thm:FakeHarper}
to  $Z = \s^t X$ and $Y = S^{2m+1}$
to deduce condition (3):
$
\pi_t( \map_*(X, S^{2m+1})) \cong
[\s^t X, S^{2m+1}]  = *.
$

\begin{proof}[Proof of Theorem \ref{thm:FakeHarper}]
According to \cite[Thm. 4.2]{MR764593}, the 
natural map $Y \to Y^\smsh_p$ induces 
a bijection $[ Z, Y]  \cong
[Z, Y^\smsh_p]$, so it suffices
to show   $[Z , Y^\smsh_p ] = *$.
Since $H^*(Y) = U(M)$, $Y$ has a Massey-Peterson 
tower, whose homotopy limit is $Y^\smsh_p$.
  Let $f_s$ be the composite $Z\to Y \to Y_s$; we will 
show by induction that $f_s \simeq *$ for all $s$.

Since $Y_0$ is a GEM, 
$f_0$ is determined by its effect on cohomology;
and since
$
\Hom_\U ( M, \twdl H^*(Z)) = 
\Ext_\U^0 (M , \s^{0} \twdl H^*(Z)) = 0
$,
  $f_0$ is trivial on cohomology, and hence trivial.
 Inductively,
suppose $f_{s-1}$
is trivial.  We have the following situation
\[
\xymatrix{
K ( \om ^s P_{s-1}) \ar@(ru,lu)[rr]^-{K(\om^s d_{s})}
\ar[r]
&
\om Y_{s-1}\ar[r] 
\ar@(rd,ld)[rr]_-{*}
&
K( \om^s P_s) \ar[r] \ar@(ru,lu)[rr]^-{K(\om^s d_{s+1})}
&Y_s\ar[d]\ar[r] 
& K(\om^s P_{s+1} )
\\
&&&
Y_{s-1} .
}
\]
Now apply  $[ Z, \? ]$ to this diagram and observe that
Lemma \ref{lem:Maps2KF}, together with 
the  isomorphism 
$\Hom_\U( \om^s P, H) \cong \Hom_\U( P, \s^s H)$
(with $H =\twdl  H^*( Z)$),
gives
\[
\xymatrix{
&& [Z, Y_{s+1}]\ar[d]
\\
\Hom_\U( 
P_{s-1}, \s^s H) 
\ar[r]^-{d_{s}^*}
\ar@(rd,ld)[rr]_{*}
&
\Hom_\U (
P_s, \s^s H ) \ar[r]^-\alpha \ar@(ru,lu)[rr]^(.7){ d_{s+1}^*} 
&[Z,Y_s]\ar[d]\ar[r]^-\beta
& \Hom_\U( 
P_{s+1} , \s^s H)
\\
&&
[Z, Y_{s-1}] .
}
\]
Exactness at $[Z, Y_s]$ implies that the homotopy class
$[f_s]$ is equal to $\alpha( g_s)$ for some 
$g_s\in \Hom_\U (P_s, \s^s H )$.
Since $[f_s]$ is in the image of the vertical map from 
$[Z,Y_{s+1}]$, it is in the kernel of $\beta$, so 
$d_{s+1}^*(g_s) = \beta( [f_s]) = 0$; in other words, 
$g_s$ is a cycle representing an element   
$[g_s] \in \Ext_\U^s(M, \s^s H)$.   Since
$ \Ext_\U^s(M, \s^s H) = 0$, we 
conclude that $g_s = d_s^*(g_{s-1})$ and so 
$
[f_s] = \alpha ( d_s^*(g_{s-1}) ) =  [ *]
$.

Since every map $f: Z\to Y$
is trivial on composition to $Y_s$ for each $s$, the 
  exact sequence
$
* \to \limone   [ Z, \om Y_s]
\longrightarrow  [  Z, Y^\smsh_p ] 
\longrightarrow  \lim [  Z, Y_s] \to *
$
reduces to an isomorphism 
$
[  Z, Y^\smsh_p ]
\cong
\limone   [ Z, \om Y_s]
$, 
and Lemma \ref{lem:CompactLimOneVanish}
finishes the proof.
%
\end{proof}

\section{(3) implies (4)}

%

The statement that (3) implies (4)  is very similar to 
Corollary 11 from \cite{MR2029919}, 
which is proved using only classical results:
 the splitting of a product after suspension;
the James construction \cite{MR0073181}; 
the Hilton-Milnor theorem in the form
proved by Gray \cite{MR0281202};
a result of Ganea on the homotopy type of the suspension
of a homotopy fiber \cite[Prop. 3.3]{MR0179791}
(see also Gray's paper \cite{MR0334198}); and a
Blakers-Massey-type theorem for
$n$-ads due to  Barratt and Whitehead \cite{MR0085509}
and  Toda \cite{MR0075589}.

The   difference is that   here   we restrict 
our attention
  to odd spheres.    To get this stronger statement, 
write $\R = \{ K \st \map_*(X, K) \sim *\}$, and
suppose that $\R$ contains all odd-dimensional spheres 
of sufficiently high dimension;  to make the argument of
\cite{MR2029919} work, it  suffices to   show that $\R$
contains all finite-type wedges of odd-dimensional spheres of 
sufficiently high dimension.    

If $W = V_1 \wdg V_2$ where both $V_1$ and $V_2$ are
wedges of odd-dimensional spheres, then the homotopy
fiber of the quotient map 
$q: W\to V_2$ is 
\[
V_1 \rtimes \om V_2 \simeq V_1 \smsh  U, 
\]
where $U$ is a wedge of even-dimensional spheres.
Now an examination of the proof 
of \cite[Prop. 7]{MR2029919} reveals that if the initial wedge 
$V_0$ is a wedge of odd-dimensional spheres, then so are all of the 
later wedges $V_n$, and so the argument carries through 
unchanged.

\section{(4) implies (5)}

%
%

Write $W = \bigvee_{i=1}^\infty S^{n_i}$, and let
Let $f: \s^t X \to W$.  Since $X$ has finite type, 
 $(\s^t X)_k$ is compact, and so 
$f((\s^t X)_k)$ is contained in a finite subwedge $V \sseq W$.
The
homotopy commutative diagram 
\[
\xymatrix{
&& (\s^t X)_k\ar[d]\ar@/^/[rrdd]^{f|_{(\s^t X)_k}}
\\
&& \s^t X \ar[d]^{*} \ar@/_/[lld]_f
\\
W \ar[rr]^-q && V \ar[rr]^-{i} && W, 
}
\]
in which $q$ is the collapse map to $V$ and $i$ is the inclusion,
shows that $f|_{(\s^t X)_k}\simeq *$, and hence that 
$f$ is a phantom map.  

The   conclusion $f\simeq *$   follows from
taking $Z = \s^t X$ in 
  Proposition \ref{prop:BLA}.
  
\begin{prop}
\label{prop:BLA}
If $Z$ is rationally trivial and of finite type, 
then 
\[
\Ph \left(\mbox{$Z, \bigvee_{i=1}^\infty S^{n_i} $} \right) = *  .
\]
\end{prop}

\begin{proof}
 The Hilton-Milnor
 theorem implies that there is a \textit{weak} product 
 of spheres 
 $P = \prod_\alpha S^{m_\alpha}$
 such that 
 $
 \om \left(  \bigvee_1^\infty S^{n_i}  \right) \simeq \om P
 $
 (that is, $P$ is the (homotopy) colimit of the diagram of finite subproducts
 of the categorical product).
 By Proposition \ref{prop:SameLoop}, 
  it suffices to show that $\Ph(Z, P) = *$.  
 
Since the skeleta of $Z$ are compact, 
every map $\s Z_k \to P$ factors 
 through a 
finite subproduct of $P$, so 
  $[\s Z_k , P]$ is a weak product 
 $\prod_\alpha [\s Z_k, S^{m_\alpha}]$.
  
Since $Z$ is rationally trivial, 
we have $\limone  [ \s Z_k, S^{m}] = *$ for each $m$, 
and since these are towers of countable groups, they are all
Mittag-Leffler.  
Write $\lambda (n,m)$ for the first $k$ for which the images
 \[
\im\left( 
 [(\s Z)_{n+k}, S^m]
\to   [(\s Z)_{n}, S^m]
   \right)
 \]
stabilize.  Since  $\lambda( n,m) = 0$
 for $m > n+1$,    the set $\{ \lambda (n,m)\st m\geq 0\}$ 
 is finite, and we define $\kappa(n)$ to be its maximum.
%
%
Now it is clear that the images 
\[
\im\left( 
\mbox{$
 \prod_\alpha [\s Z_{n+k} , S^{m_\alpha}]
 \to
  \prod_\alpha [\s Z_n, S^{m_\alpha}]
  $}
  \right)
  \]
  are independent of $k$ for $k \geq \kappa(n)$.
  Thus the tower $\{   \prod_\alpha [\s Z_k, S^{n_\alpha}]\}$
  is Mittag-Leffler,  $\Ph( Z, P) = *$,  and  the proof is complete.
%
%
%
%
%
%
%
%
\end{proof}

%
%
%
%
%
%
%
%

\section{(5) implies (6)}

Let $f: \s^t X\to K$, where $K$ is a 
simply-connected
finite-dimensional CW complex, which we may 
assume has trivial $1$-skeleton.
  Since $X$ is of finite type, the image $f( \s^t X_k )$
of each skeleton    is contained in a finite
subcomplex of $K$, and hence $f$ 
factors through the inclusion of   a countable 
subcomplex of $K$, which is necessarily  
simply-connected and finite-dimensional.  Thus to show 
$f \simeq *$, it suffices
to prove that $\map_*( X, L) \sim *$ for all 
finite-dimensional countable
 simply-connected CW complexes $L$.

 Write   $\A = \{  \bigvee_{i=1}^\infty  S^{n_i} \}$ (with each $n_i > 1$). 
Since $\A$ is closed under suspension and smash product, 
and $\A\sseq \R = \{ K \st \map_*(X,K)\sim *\}$, 
  Theorem 8 of 
\cite{MR2029919} implies that every simply-connected
space which has finite cone length with respect 
to 
$\A = \A^\wdg$ is also in $\R$.
But this includes all simply-connected
  countable finite-dimensional CW complexes.

\section{(6) implies (7)}

We continue to write $\R = \{ K \st \map_*(X,K)\sim *\}$. 
We know that $\R$ contains all simply-connected
finite-dimensional wedges of spheres, a collection 
that is closed under suspension and smash product.
Since generic wedges $\bigvee S^{n_\alpha}$ are 
finite-type wedges of such wedges, 
Theorem 8 of \cite{MR2029919} implies 
that $\R$ contains $\WedgesOfSpheres$. 
The same theorem now implies that 
every simply-connected space 
$X$ with $\cl_{\WedgesOfSpheres}(X) < \infty$
is also in $\R$.

\section{Non-Simply-Connected Targets}

Suppose, finally,  that $K$ is not simply-connected and that 
$\pi_1(X)$ has no perfect quotients.
If $f: \s^t X\to K$ is trivial on fundamental groups, then 
there is a lift in the diagram
\[
\xymatrix{
&& \twdl K \ar[d]^p
\\
\s^t X \ar[rr]^-f\ar@/^/@{..>}[rru] && K
}
\]
where $p$ is the universal cover of $K$.  Since $K$
is finite-dimensional, so  is $\twdl K$, and 
so $f \simeq *$.  On  the other hand, if $f$ is nontrivial 
on fundamental groups, then we write $G = \im ( f_*)$
and consider the covering space $q: L \to K$ 
corresponding to the subgroup $G \sseq \pi_1(K)$.
Again, we have a lift in the diagram
\[
\xymatrix{
&& L \ar[d]^q
\\
\s^t X \ar[rr]^-f\ar@/^/@{..>}[rru]^\phi  && K
}
\]
and $\phi$ induces a surjection on 
fundamental groups.  But since $G$ is not perfect, there
is a nontrivial map $u: L \to K(A,1)$ (which must be nonzero on 
fundamental groups)
for some abelian group $A$.
Thus $\phi$ is nonzero on cohomology and so
 $\s \phi : \s^{t+1} X\to \s L$ is nontrivial.   

To finish the proof, we observe that if $L \to K$ is a 
covering with $L$ path-connected, then forming the pullback
squares
\[
\xymatrix{
L_0 \ar[d]\ar[r] & L_1\ar[d] \ar[r] & \cdots \ar[r] & L_{n-1}\ar[d]\ar[r]
& L_n \ar@{=}[r] \ar[d] & L\ar[d]
\\
K_0 \ar[r] & K_1 \ar[r] 
& \cdots \ar[r] & K_{n-1}\ar[r] & K_n \ar@{=}[r] & K
}
\]
over a $\WedgesOfSpheres$-cone decomposition of $K$ yields
a $\WedgesOfSpheres$-cone decomposition of $L$.
%
%
%
%
Thus the nontriviality of $\s \phi$
   contradicts Theorem \ref{thm:main}(7) 
because $\s L$ is simply-connected and 
$\cl_{\WedgesOfSpheres}(\s L) 
 \leq \cl_{\WedgesOfSpheres}(L) <  \infty$.

\end{document}